\documentclass[12pt]{article}
\usepackage{amsmath,amssymb,amsthm}
\usepackage[pdftex]{graphicx}
\usepackage{etoolbox}
\usepackage[hidelinks]{hyperref}
\hypersetup{pdftitle={Homotopy groups and quantitative Sperner-type lemma},pdfauthor={Oleg R. Musin},pdfsubject={Quantitative Sperner bounds from boundary homotopy classes}}
\newtheorem{thm}{Theorem}[section]
\newtheorem{cor}[thm]{Corollary}
\newtheorem{lemma}[thm]{Lemma}
\newtheorem{prop}[thm]{Proposition}

\newtheorem{conj}[thm]{Conjecture}
\newcommand{\KeepStatementTogether}[1]{%
  \BeforeBeginEnvironment{#1}{\par\begin{samepage}}%
  \AfterEndEnvironment{#1}{\end{samepage}}}
\forcsvlist{\KeepStatementTogether}{thm,cor,lemma,prop,conj}
\numberwithin{equation}{section}
\DeclareMathOperator{\conv}{conv}
\DeclareMathOperator{\itr}{int}
\DeclareMathOperator{\lk}{lk}
\DeclareMathOperator{\st}{st}
\DeclareMathOperator{\vrt}{Vert}
\DeclareMathOperator{\sgn}{sign}
\DeclareMathOperator{\Arf}{Arf}
\newcommand{\Z}{\mathbb Z}
\newcommand{\R}{\mathbb R}
\title{Homotopy groups and quantitative Sperner--type lemma}
\author{Oleg R. Musin}
\date{}
\begin{document}
\maketitle

\begin{abstract}
Sperner's lemma guarantees a fully labeled simplex under a boundary labeling rule. We ask how many fully labeled simplices are forced by the homotopy class of the boundary map. For a geometrically triangulated convex ball, we compare this number with the least number of simplices over a target facet among representatives of the boundary class on smoothly triangulated spheres. The comparison holds in equal and consecutive dimensions, and whenever the number of fully labeled simplices is at most twice the number of colors. A pyramid construction gives a comparison with the suspended class in all dimensions.

If the boundary sphere has dimension one or two greater than the target sphere, a nonzero boundary class forces at least nine fully labeled simplices. We construct geometric examples attaining nine when the difference is one. When the difference is two, we prove that the spherical minima eventually become constant as the dimensions increase. We also show that the choice of triangulation matters: the smooth Hopf minimum is nine for every nonzero Hopf invariant, whereas for convex polytopal spheres the minimum grows as the square root of its absolute value. The proofs use framed fibres, monochromatic cycles, and maps on vertex links.
\end{abstract}


\section{Introduction}\label{sec:intro}

Sperner's lemma guarantees a fully labeled simplex under a boundary labeling rule \cite{Sperner}. A more general obstruction comes from the map defined by the labeling: if its restriction to the boundary is not null-homotopic, then every extension has a fully labeled simplex \cite{Mus16,Mus17}. How many such simplices must there be? We study this question by comparing labelings of balls with simplicial maps between spheres. The main point is to relate the number of fully labeled faces to a minimum depending only on the boundary homotopy class.
\subsection{Triangulations and the spherical minimum}

We use three notions of triangulation. A \emph{smooth triangulation} of a smooth manifold $X$ is a finite simplicial complex $T$ with a homeomorphism $|T|\to X$ whose restriction to each closed simplex extends to a smooth embedding of a neighbourhood in its affine span. For a manifold with boundary the extension is taken in a smooth double. A \emph{geometric triangulation} of a convex polytope triangulates the entire polytope by straight simplices meeting face to face; it may have interior vertices and may subdivide the boundary facets. A \emph{convex polytopal sphere} is the boundary complex of a convex simplicial polytope. This last requirement is stronger than being a geometric subdivision of a polytope boundary. Radial projection from an interior point gives either boundary a smooth realization on the standard sphere.

Put $S^n_{n+2}=\partial\Delta^{n+1}$, with vertices $0,\ldots,n+1$. The subscript records the number of vertices; we also write $\partial\Delta^{n+1}$, or $S^n$ when only the underlying sphere matters. A labeling $L:\vrt(T)\to\{0,\ldots,n+1\}$ defines an affine map $f_L:|T|\to\Delta^{n+1}$ on each simplex. Its image lies in $S^n_{n+2}$ precisely when no simplex contains all $n+2$ labels; see Figure~\ref{fig:simplicial-map}. We call an $(n+1)$-simplex with all these labels \emph{fully labeled}.

If $f:T\to Q$ is simplicial and $s$ is an $n$-simplex of $Q$, write
\[
 \mu(f,s)=\#\{\tau\in T:\dim\tau=n,\ f(\tau)=s\}.
\]
For $a\in\pi_m(S^n)$, define
\begin{equation}\label{eq:mu-definition}
 \mu(m,n,a)=\min_{(T,L),\,s}\mu(f_L,s),
\end{equation}
where $T$ smoothly triangulates the standard oriented $S^m$, the map $f_L:T\to S^n_{n+2}$ represents $a$, and $s$ runs through the target facets. The set of representatives is nonempty by smooth triangulation and simplicial approximation \cite{Wh40,Hat}. Each source $n$-simplex is counted once, regardless of how many maximal simplices contain it. Thus these are maximal faces when $m=n$ and lower-dimensional faces when $m>n$. A constant map gives $\mu(m,n,0)=0$. Reflection of the source shows that the minimum is unchanged under $a\mapsto-a$.

The restriction of \eqref{eq:mu-definition} to convex polytopal spheres defines $\mu^*(m,n,a)$. Similarly, $\lambda(m,n,a)$ and $\lambda^*(m,n,a)$ minimize the number of source vertices in the two classes. The starred minima are finite. Indeed, sufficiently many derived subdivisions of a PL sphere are polytopal \cite{AI15}, and each subdivision admits a labeling carried by the old map. In particular,
\[
 \mu\le\mu^*,\qquad \lambda\le\lambda^*.
\]
When $m=n$, we identify the class with its degree $d$ and write $\mu(n,d)$ and $\lambda(n,d)$. When $(m,n)=(3,2)$, we use the Hopf invariant $d\in\Z$ and abbreviate the simplex minima to $\mu(d)$ and $\mu^*(d)$.

\begin{figure}[htbp]
\centering
\includegraphics[width=0.72\textwidth]{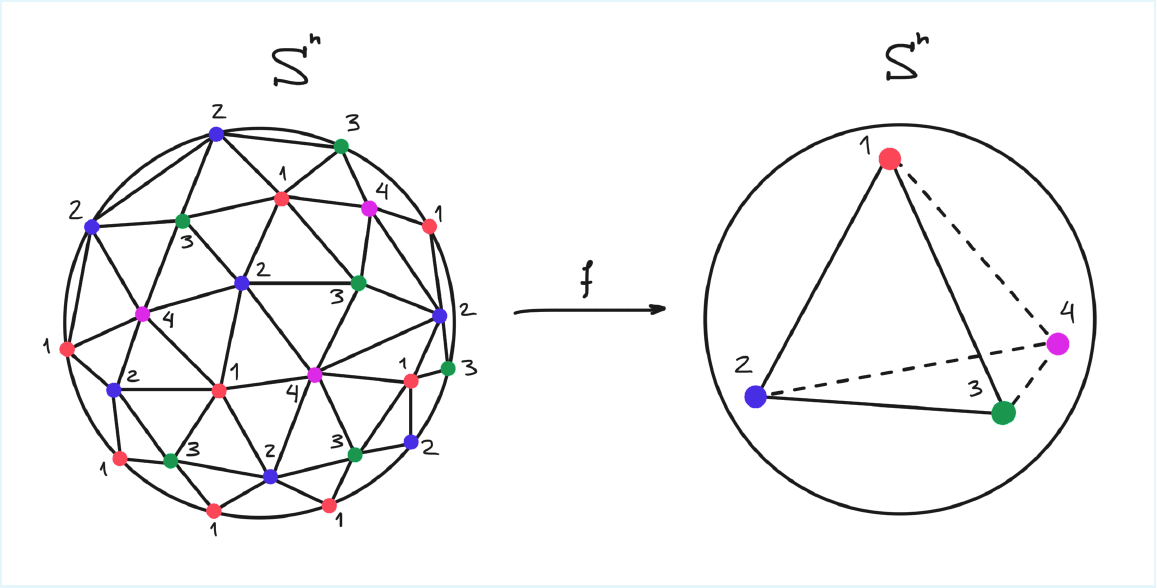}
\caption{A labeling defines a simplicial map to $\partial\Delta^{n+1}$; the illustration shows $n=2$.}
\label{fig:simplicial-map}
\end{figure}

\subsection{Results}

\begin{thm}[Degree and consecutive dimensions]\label{th:spherical}
For $n\ge1$ and $d\in\Z$,
\[
 \mu(n,d)=\mu^*(n,n,d)=|d|.
\]
For $n\ge2$ and every nonzero $a\in\pi_{n+1}(S^n)$,
\[
 \mu(n+1,n,a)=9.
\]
The latter group is $\Z$ for $n=2$ and $\Z_2$ for $n\ge3$.
\end{thm}

The Hopf equality was proved in our joint paper with Bludov and Fomin \cite[Theorem~2.3]{BFM25}; we include a proof in the present notation. We next consider the groups $\pi_{n+2}(S^n)$, including the cases $n=2,3$ before the stable range. Write $\eta_2:S^3\to S^2$ for the Hopf class and $\eta_n=\Sigma^{n-2}\eta_2$ for $n\ge2$. The unique nonzero class in $\pi_{n+2}(S^n)\cong\Z_2$ is
\begin{equation}\label{eq:surface-groups-notation}
 a_n=\eta_n\circ\eta_{n+1},\qquad
 \mu_n=\mu(n+2,n,a_n)\qquad(n\ge2).
\end{equation}

\begin{thm}[The groups $\pi_{n+2}(S^n)$]\label{th:surface-groups}
For every $n\ge2$,
\begin{equation}\label{eq:surface-groups-bounds}
 9\le\mu_n\le7(n+1)^2.
\end{equation}
The sequence $(\mu_n)$ is eventually constant.
\end{thm}

For the lower bound, we first analyze surface fibres with at most eight vertices in the case $\pi_4(S^2)$, then use vertex links to pass to higher dimensions. The upper bound comes from a normal tube around a seven-vertex framed torus; for $n=2$ it gives $9\le\mu_2\le63$. A different argument gives a uniform bound on all $\mu_n$ and proves eventual constancy (Proposition~\ref{prop:surface-groups-stabilization}). We do not determine the constant value or the dimension at which it is reached.

For a labeled triangulation $(T,L)$, let $C(T,L)$ denote the number of fully labeled $(n+1)$-simplices. The main comparison is the following.

\begin{thm}[Quantitative Sperner lemma]\label{th:sperner}
Let $m\ge n\ge1$, let $T$ be a geometric triangulation of a compact convex $(m+1)$-polytope $P$ with nonempty interior, and label its vertices by $0,\ldots,n+1$. Suppose that no boundary $(n+1)$-simplex is fully labeled, so that the boundary map defines a class $a=[\partial f_L]\in\pi_m(S^n)$. Then
\begin{equation}\label{eq:sperner-suspended}
 C(T,L)\ge\mu(m+1,n+1,\Sigma a).
\end{equation}
Moreover,
\begin{equation}\label{eq:sperner-comparison}
 C(T,L)\ge\mu(m,n,a)
\end{equation}
whenever at least one of the following holds:
\begin{enumerate}
\item $m=n$;
\item $m=n+1$;
\item $C(T,L)\le2(n+2)$.
\end{enumerate}
\end{thm}

To prove \eqref{eq:sperner-suspended}, we use the triangulated ball as the base of a convex pyramid. The labeled boundary represents the suspension and has exactly the same count. In \eqref{eq:sperner-comparison}, the case $m=n$ is the signed degree count \cite[Theorem~3.9]{Mus16}. The other cases require information about the fibres of the labeling map.

\begin{cor}[The lower bound nine]\label{cor:nine}
Under the hypotheses of Theorem~\ref{th:sperner}, let $n\ge2$ and $a\ne0$. If $m=n+1$ or $m=n+2$, then
\[
 C(T,L)\ge9.
\]
For $m=n+1$ this bound is attained by geometric examples for every $n\ge2$.
\end{cor}

For the groups $\pi_{n+2}(S^n)$, suspension preserves the nonzero class, and \eqref{eq:sperner-suspended} gives $C(T,L)\ge\mu_{n+1}\ge9$ without any restriction on the count. In particular, a nonzero class in $\pi_4(S^2)$ forces at least nine fully labeled tetrahedra in a convex five-ball. We do not know whether nine is attained when $m=n+2$. For the first case we conjecture:
\begin{conj}\label{conj:nine}
For $0\ne a\in\pi_4(S^2)$, one has $\mu(4,2,a)=9$.
\end{conj}
Together with Corollary~\ref{cor:nine}, this conjecture would give $C\ge\mu$ for $(m,n)=(4,2)$ without a restriction on $C$. More generally, is $\mu_n=9$ for every $n\ge2$? Proposition~\ref{prop:surface-groups-stabilization} proves $C\ge\mu_n$ for all sufficiently large $n$, without assuming an answer to this question.

The restriction to convex polytopal spheres gives a different Hopf minimum:
\[
 \mu^*(d)\ge6+\lceil3\sqrt{|d|}\rceil,\qquad
 \mu^*(d)\asymp\sqrt{|d|}\quad(d\ne0).
\]
There is also the bound $\mu^*(d)\ge12$ for $|d|\ge2$; see Theorem~\ref{th:polytopal}. Theorem~\ref{th:sperner} uses $\mu$; its spherical representatives need not be convex polytopal.

Section~\ref{sec:fibres} collects the lemmas about fibres and their graphs. In Section~\ref{sec:degree-hopf} we discuss degree, vertex counts, and the two Hopf minima. Section~\ref{sec:short} proves the bounds for $\pi_{n+1}(S^n)$ and $\pi_{n+2}(S^n)$. Section~\ref{sec:balls} contains the proof of the quantitative Sperner lemma and the eventual stabilization result.

\section{Fibres and their graphs}\label{sec:fibres}

In this section $f:T\to Q$ is simplicial, $T$ is a smooth triangulation of a manifold or a geometric triangulation of a convex polytope, and $\sigma$ is a maximal $r$-simplex of $Q$. Write $\dim T=k+r$ and choose $y\in\itr\sigma$. If the source has boundary, we call a fibre closed when it is disjoint from the boundary. We use subdivisions to describe neighbourhoods and links, but count faces in the original triangulation $T$.

\subsection{Cells, products, and color cycles}

\begin{lemma}[Fibre product]\label{lem:product}
The fibre $M=f^{-1}(y)$ is a PL $k$-manifold, with boundary $M\cap\partial T$. Its cell in a simplex $\tau$ mapping onto $\sigma$ is
\begin{equation}\label{eq:cell}
 F_\tau\cong\prod_{i=0}^r\Delta^{|V_i(\tau)|-1},\qquad
 V_i(\tau)=\{v\in\vrt(\tau):f(v)=e_i\}.
\end{equation}
For a sufficiently small closed affine simplex $B\subset\itr\sigma$ containing $y$ in its interior, there is a boundary-compatible PL homeomorphism
\begin{equation}\label{eq:product}
 h:M\times B\longrightarrow f^{-1}(B),\qquad
 f h(x,b)=b,\quad h(x,y)=x.
\end{equation}
The vertices of $M$ correspond bijectively to the source $r$-faces mapping bijectively onto $\sigma$.
\end{lemma}

\begin{proof}
In the affine coordinates $x=\sum_v t_vv$ of $\tau$, the fibre equations are
\[
 t_v\ge0,\qquad \sum_{v\in V_i(\tau)}t_v=y_i\quad(0\le i\le r).
\]
They give \eqref{eq:cell}; its vertices are $\sum_i y_iv_i$, with one $v_i\in V_i(\tau)$ for each color. These choices are exactly the counted $r$-faces.

Triangulate $B$ with $y$ as a vertex. For a simplex $\beta$ in this triangulation, the polytope $\tau\cap f^{-1}(\beta)$ has the face lattice of $F_\tau\times\beta$. Indeed, a face is specified independently by a face of $\beta$ and nonempty subsets of each $V_i(\tau)$; all target coordinates remain positive. Its vertices are $\sum_i b_iv_i$, with $b\in\vrt(\beta)$. These correspondences agree on common faces.

Subdivide both cell complexes barycentrically, using the arithmetic means of their vertices, and map corresponding barycentres affinely. This gives a PL homeomorphism. Each vertex $b$ of a base face occurs equally often among the images of the vertices of the corresponding source cell. Therefore its barycentre maps to the barycentre of that base face, proving $fh=\operatorname{pr}_B$. Over $y$ the map is the identity, and the construction respects the source boundary.

Finally, $M\times\itr B$ is open in the ambient PL manifold. In a product triangulation the link of a top-dimensional base simplex over a fibre vertex is the link of that vertex in $M$. The ambient link criterion makes these links PL spheres, or balls at the boundary. This proves the manifold assertion.
\end{proof}

The rescaling $t_v\mapsto t_vb_{L(v)}/y_{L(v)}$ also gives a continuous, face-preserving product, though it is not generally PL. The subdivision argument proves the PL assertion. Both constructions use the positivity of the target coordinates and apply only over the interior of $\sigma$.

A vertex of the fibre has a \emph{tuple} $(v_0,\ldots,v_r)$ of source vertices, one of each color. Fibre edges change exactly one coordinate. A coordinate is \emph{constant} on a component if the same source vertex occurs in every one of its tuples. In dimension $k=1$, a component is a circle; every edge comes from a simplex with one repeated color, called its \emph{pivot color}. If $m_i(\gamma)$ counts the edges with pivot color $i$, then
\begin{equation}\label{eq:circle-count}
 \mu(f,\sigma)=\sum_\gamma |\vrt(\gamma)|,
 \qquad |\vrt(\gamma)|=\sum_{i=0}^r m_i(\gamma).
\end{equation}
Every nonconstant coordinate on a circle changes at least twice.

\begin{lemma}[Color retraction]\label{lem:colors}
Let $K_i$ be the monochromatic subcomplex of color $i$. On $\{f_i>0\}$ there is a continuous retraction
\begin{equation}\label{eq:color-retraction}
 R_i(x)=\frac{\sum_{L(v)=i}t_vv}{f_i(x)},\qquad
 H_i(x,s)=(1-s)x+sR_i(x).
\end{equation}
For $x\in f^{-1}(y)$, one has $f(H_i(x,s))=(1-s)y+se_i$. If $\gamma$ is an oriented fibre circle, $(R_i)_*[\gamma]$ is an integral edge cycle $Z_i(\gamma)$ with $m_i(\gamma)$ edge occurrences. In particular,
\begin{equation}\label{eq:short-cycle}
 m_i(\gamma)\le2\quad\Longrightarrow\quad Z_i(\gamma)=0.
\end{equation}
Homotopies whose target segments are disjoint have disjoint images.
\end{lemma}

\begin{proof}
The formulas agree on common faces and stay in the simplex containing $x$. An $i$-pivot traverses its monochromatic edge once, while every other pivot collapses to a vertex. A closed edge walk of length at most two in a simplicial graph has zero associated chain: there are no loops or multiple edges. The last assertion follows by applying $f$.
\end{proof}

Oriented affine coordinates on the target simplex give transverse coordinates in the product of Lemma~\ref{lem:product}. We use the resulting normal framing throughout the paper. We use framed fibres only in dimensions one and two. After smoothing them and their product neighbourhoods, the transverse coordinates give the framing of a regular fibre of a smooth approximation of $f$. The Pontryagin--Thom construction then identifies its framed class with the class of $f$ \cite{Milnor,Putman}. Smoothing is used only for this identification; all counts refer to the original PL fibre.

\subsection{Vertex links and normal tubes}

\begin{lemma}[Smooth links]\label{lem:smooth-link}
The link of an interior vertex in a smooth triangulation of a $d$-manifold has a smooth realization on the standard $S^{d-1}$ with the same simplices. The same holds for its geometric subdivisions. For a geometric triangulation this realization is radial projection about the vertex.
\end{lemma}

\begin{proof}
In a smooth coordinate chart at the vertex, differentiate the simplex embeddings there. The resulting linear maps on their tangent cones agree on common faces and are nondegenerate. Orient maximal cones by the chart; all local degrees are positive. On a sufficiently small link the piecewise linear tangent map is uniformly close, after rescaling, to the original chart map, so its degree is one. Positive local multiplicities then exclude overlapping cone interiors and show that the cones form a complete fan. Radial projection of its link, or of any geometric subdivision, is a smooth embedding on each closed simplex.
\end{proof}

\begin{lemma}[Relative completion]\label{lem:completion}
Let $Q_0\subset S$ be a subcomplex labeled with $r$ colors. Let $\mathcal A$ be its $(r-1)$-faces using all these colors, and assume that every coface in $S$ of every member of $\mathcal A$ lies in $Q_0$. A subdivision relative to $Q_0$, labeled with one additional color $c$ outside it, defines a map to $\partial\Delta^r$. The faces over the facet opposite $c$ are exactly $\mathcal A$.
\end{lemma}

\begin{proof}
Stellarly subdivide the faces outside $Q_0$ in decreasing dimension. Give the new vertices and the original vertices outside $Q_0$ color $c$. In each resulting simplex the vertices of other colors form a face of $Q_0$. If that face contains all $r$ old colors, it contains a member of $\mathcal A$, whose original cofaces all lie in $Q_0$ and were not subdivided. Such a simplex has no vertex of color $c$. This proves both claims. Barycentres can be used in a geometric realization.
\end{proof}

\begin{lemma}[The vertex-link model]\label{lem:link-model}
Assume $r\ge2$. Suppose that an interior source vertex $v$ of color $c$ is constant on a union $M_v$ of components of a closed fibre over $y\in\itr\sigma$. Let $C_v$ be their number of fibre vertices. There is a relative subdivision of $\lk_T(v)$ and a map
\[
 h_v:S^{k+r-1}\longrightarrow\partial\sigma\cong S^{r-1}
\]
with exactly $C_v$ faces over the facet opposite $c$. Its fibre is the radial projection of $M_v$. On a neighbourhood of the retained fibre cone the original map is
\begin{equation}\label{eq:cone-formula}
 f(tv+(1-t)x)=te_c+(1-t)h_v(x).
\end{equation}
For $k=1,2$ and smooth spherical sources, the stable framed class of $M_v$ is, with compatible orientations, that of this link fibre with one normal direction added.
\end{lemma}

\begin{proof}
Delete $v$ from the counted $r$-faces of $M_v$, obtaining a family $\mathcal A_v$ in the link, and let $Q_v$ be the union of their closed stars. A coface of a counted face has a connected product fibre by \eqref{eq:cell}, so all its fibre vertices belong to the same selected component. It contains no second vertex of color $c$: otherwise one of these tuples would omit $v$. Thus $Q_v$ has only the other $r$ colors. Its faces with all those colors are precisely $\mathcal A_v$, and all their cofaces were retained. Lemmas~\ref{lem:completion} and~\ref{lem:smooth-link} give $h_v$ with the asserted count.

Projecting along the rays from $v$ removes the constant coordinate $y_c$ and divides the others by $1-y_c$. This identifies the fibres and proves \eqref{eq:cone-formula}. In a product neighbourhood its normal coordinates split into the link coordinates and the radial coordinate. After smoothing, the extra radial normal direction gives the stated relation between the stable framings.
\end{proof}

\begin{lemma}[A labeled normal tube]\label{lem:tube}
Let $k\in\{1,2\}$ and $r\ge2$. Suppose a closed framed $N^k\subset S^{k+r}$ has a smooth triangulation $K$, with face numbers $f_j(K)$. The Pontryagin--Thom class of $N$ has a labeled representative with
\begin{equation}\label{eq:tube-count}
 \sum_{j=0}^k\binom rj f_j(K)
 =f_0(K)+r f_1(K)+\binom r2 f_2(K)
\end{equation}
source $r$-faces over a chosen target facet, with $f_2=0$ when $k=1$.
\end{lemma}

\begin{proof}
Take an inner normal $r$-simplex $A=[p_1,\ldots,p_{r+1}]$ and surround it by a collar triangulated as prisms over its facets, using a common vertex order. Label $p_i$ by $i$ and all outer-boundary vertices by $0$. Every collar simplex omits at least one label from $1,\ldots,r+1$. This gives a degree-one map
\[
 (D^r,\partial D^r)\longrightarrow(\partial\Delta^{r+1},e_0)
\]
with exactly one $r$-simplex over $[1,\ldots,r+1]$.

Use the ordered staircase triangulation of $K\times D^r$ and give a product vertex the label of its normal coordinate. A counted $r$-simplex has vertices
\[
 (v_1,p_1),\ldots,(v_{r+1},p_{r+1}),\qquad
 v_1\le\cdots\le v_{r+1}.
\]
If its minimal supporting base face has dimension $j$, each of that face's $j+1$ vertices must occur. The positive multiplicities sum to $r+1$, giving $\binom rj$ possible sequences. Summing over the base faces proves \eqref{eq:tube-count}. No collar simplex contributes to this count.

Transport the product triangulation by the tubular parametrization specified by the framing. The resulting simplices are smooth embeddings. In a monochromatic outer collar, round the boundary and extend the smooth triangulation to the complement relative to the central tube \cite{Wh40}. Give every complementary vertex label $0$. No counted face is added, and a regular fibre is a parallel copy of $N$ with the prescribed framing.
\end{proof}

\begin{lemma}[Addition]\label{lem:addition}
For $m>n\ge2$ and finitely many classes $a_j\in\pi_m(S^n)$,
\[
 \mu\left(m,n,\sum_j a_j\right)\le\sum_j\mu(m,n,a_j).
\]
More precisely, representatives of the $a_j$ can be joined without changing the sum of their counts.
\end{lemma}

\begin{proof}
Move the chosen target facets to $[1,\ldots,n+1]$ by even permutations of the colors; these preserve the classes. Discard zero classes. Each remaining map has a vertex of color $0$, since otherwise its image is contractible. Subdivide the edges incident to that vertex, giving the new vertices color $0$. This produces an all-zero top simplex without subdividing or creating a counted face. The new map is carried by the old target faces and hence homotopic to it. Remove such simplices and form oriented connected sums. The maps agree on the gluing boundaries and represent the sum. Rounding the gluing collars realizes the complex smoothly on the standard sphere, with no additional counted face.
\end{proof}

\subsection{Tuple graphs}

The graph of a connected closed fibre of positive dimension is simple and two-connected. Its one-skeleton is connected because the cells in \eqref{eq:cell} are convex polytopes. In dimension one it is a cycle. In dimension at least two, deleting one point leaves the manifold connected; a path avoiding a fibre vertex can be moved into the cell one-skeleta while still avoiding that vertex. This uses the elementary fact that deleting one vertex from the graph of a convex polytope of dimension at least two leaves it connected. Moreover, a fibre vertex has degree at least $k+1$: its source $r$-face has a link which is a simplicial $(k-1)$-sphere, and every link vertex supplies a distinct incident fibre edge.

\begin{lemma}[Coordinate count and equality]\label{lem:graph-bound}
Let $G$ be a finite simple two-connected graph whose vertices are distinct tuples, with every edge changing exactly one coordinate. If the nonconstant coordinates take $s_1,\ldots,s_q$ values, then
\begin{equation}\label{eq:alphabet-bound}
 |V(G)|\ge\sum_{i=1}^q s_i\ge2q.
\end{equation}
If $|V(G)|=2q$, all these coordinates are binary and $G$ has at least four vertices of degree two.
\end{lemma}

\begin{proof}
\emph{The inequality.} Start an open ear decomposition with a cycle. A coordinate taking $s\ge2$ values on that cycle changes at least $s$ times. For an ear of length $\ell$, a formerly constant coordinate that now takes $s$ values changes at least $s$ times; a previously active coordinate acquiring $b$ new values changes at least $b$ times. The distinct endpoints differ in an old coordinate, where at least $b+1$ changes are required. The increase in $\sum s_i$ is therefore at most $\ell-1$, the number of new vertices on the ear. This proves \eqref{eq:alphabet-bound}.

\emph{The equality case.} Use the weaker count $2q$. An ear introducing $t$ active coordinates has at least $2t+1$ edges. Equality in the final count forces equality at each step: on the first cycle each coordinate changes twice, and on each ear each new coordinate changes twice while the old coordinates change only once in total. All values are therefore binary. Consider the graph induced by the selected vertices in the $q$-cube; it suffices to find four degree-two vertices there.

First take a cube cycle of length $2s$ changing each of its $s$ coordinates twice. Two cube chords cannot have alternating endpoints. Otherwise their endpoints, in cyclic order and after translation, would be
\[
 0,\quad b,\quad e_i,\quad b+e_j
\]
(addition modulo two). Any nonzero coordinate of $b$ outside $\{i,j\}$ would change at least four times. Hence $b$ is supported on $\{i,j\}$, but every such choice makes two of the displayed endpoints coincide. The induced graph is therefore outerplanar, with the cycle as its outer boundary. It is bipartite; quadrangulate its bounded faces by additional noncrossing diagonals. A single quadrilateral has four degree-two vertices. Otherwise two leaves of the dual tree of the quadrangulation supply two such vertices each. Removing the added diagonals preserves their degree two.

Now add an ear in the equality case. Its endpoints $u,v$ differ in one old coordinate; every internal vertex has old tuple $u$ or $v$ and a nonzero tuple of new coordinates. Hence its only possible old neighbours are $u,v$. Together with the cube edge $uv$, the ear is a cycle changing each coordinate twice. Its induced graph has four degree-two vertices, at least two internal to the ear. At most two old degree-two vertices can lose that property, so at least four remain after adding the ear. Induction proves the claim for the induced graph, and therefore for $G$, since $G$ has minimum degree two.
\end{proof}

\begin{cor}[Strict count for higher-dimensional fibres]\label{cor:strict-count}
If a connected closed fibre component $M_\alpha$ has dimension $k\ge2$ and $q_\alpha$ nonconstant coordinates, then
\begin{equation}\label{eq:strict-count}
 |\vrt(M_\alpha)|\ge2q_\alpha+1.
\end{equation}
\end{cor}

\begin{proof}
Its graph is simple and two-connected, with minimum degree at least $k+1\ge3$. Lemma~\ref{lem:graph-bound} therefore excludes equality in the bound $|\vrt(M_\alpha)|\ge2q_\alpha$.
\end{proof}

\section{Degree and Hopf classes}\label{sec:degree-hopf}

\subsection{Degree and vertex counts}

For a simplicial map of oriented $n$-spheres, comparison of the coefficient of an oriented target facet $s$ in the image of the fundamental cycle gives
\begin{equation}\label{eq:degree-count}
 \deg f=\sum_{\tau:\,f(\tau)=s}\sgn(f|_\tau).
\end{equation}
Consequently $\mu(f,s)\ge|\deg f|$. For an oriented compact $(n+1)$-manifold with boundary, let $\deg(\partial f_L)$ denote the sum of the degrees over the boundary components. The same chain argument gives
\begin{equation}\label{eq:relative-degree}
 \deg(\partial f_L)=\sum_{\tau\ \text{fully labeled}}\sgn(f_L|_\tau),
 \qquad C(T,L)\ge|\deg(\partial f_L)|.
\end{equation}
This is the degree form of the quantitative Sperner lemma \cite{Mus16}. Figure~\ref{fig:degree-three} illustrates the degree-three case.

\begin{figure}[htbp]
\centering
\includegraphics[scale=0.9]{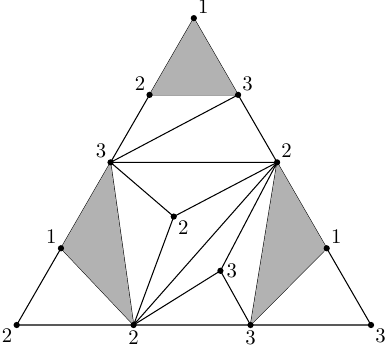}
\caption{Boundary degree three forces at least three fully labeled triangles.}
\label{fig:degree-three}
\end{figure}

\begin{lemma}[One-point suspension]\label{lem:one-point}
A map $f:\partial P\to\partial\Delta^{k+1}$, with $P$ a convex simplicial $(k+1)$-polytope, admits a convex polytopal suspension with one additional source vertex, the same degree, and the same count over a corresponding target facet.
\end{lemma}

\begin{proof}
Choose a vertex $w$, let $A$ be its deletion in $\partial P$, and put $J=\lk(w)$. Replace $w$ by $(w,-1),(w,1)$ and place all other vertices at $(u,0)$. Their convex hull has boundary
\[
 T_v=(v*A)\cup(w*A)\cup([v,w]*J)=(v*\partial P)\cup(w*A),
\]
where $v=(w,1)$ is new. To check the facets, a supporting functional with nonzero last coefficient selects exactly one lifted vertex and a facet of $P$ avoiding $w$; with zero last coefficient it selects both lifted vertices and a facet through $w$. Thus the displayed complex is simplicial and polytopal.

Keep the old labels and give $v$ the new label $k+2$. The cone $v*\partial P$ maps to the cone on the old target boundary, while $w*A$ maps into the old target simplex. Over $[k+2]*s$ the counted faces are exactly $v*\tau$ with $f(\tau)=s$. Orient source and target cones by putting their apex first. Corresponding local signs agree, so \eqref{eq:degree-count} preserves the degree; reflection adjusts the boundary orientation if necessary.
\end{proof}

\begin{proof}[Degree assertion of Theorem~\ref{th:spherical}]
For $d>0$, repeat the labels $A,B,C$ around a convex $3d$-gon. Its boundary map has degree $d$ and exactly $d$ edges over $AB$. Iterate Lemma~\ref{lem:one-point} to obtain the same count in every dimension. Together with \eqref{eq:degree-count}, this proves $\mu(n,d)=\mu^*(n,n,d)=d$. Reflection treats negative $d$, and a constant map treats zero.
\end{proof}

Minimizing the number of vertices is a different problem. The following results give a comparison. For $d>0$, Madahar and Sarkaria \cite{MS2} give
\[
 \lambda(2,d)=\lambda^*(2,2,d)=
 \begin{cases}4,&d=1,\\7,&d=2,\\2d+2,&d\ge3.\end{cases}
\]
Their source spheres are polytopal by Steinitz's theorem. For $n\ge3$ and $d=2,3,4$, our joint work with Apolonskaya \cite{AM26} gives
\[
 \lambda(n,d)=\lambda^*(n,n,d)=n+d+3.
\]
The upper examples for $d=2,3$ follow from polygons or two-spheres by Lemma~\ref{lem:one-point}; for $d=4$, the ten-vertex three-sphere is obtained by stellar subdivisions of facets of a simplex, followed by the same suspension. Also $\lambda(1,2)=\lambda^*(1,1,2)=6$. More generally, Ryabichev \cite[Theorem~1 and Remark~2]{Ry26} proves, with $q=\lfloor(n+1)/2\rfloor$,
\[
 \lambda(n,d)\asymp_n\lambda^*(n,n,d)\asymp_n d^{1/q}
 \qquad(d\to\infty).
\]
His lower bound applies to arbitrary triangulations and the upper examples are convex polytopal. Exact values beyond the known cases remain open.

\subsection{The smooth Hopf minimum}

For a simplicial map $f:S^3\to S^2$ between smooth triangulations, two distinct interior points of a target triangle have oriented link fibres, and
\begin{equation}\label{eq:hopf-link}
 H(f)=\lk(f^{-1}(p),f^{-1}(q)).
\end{equation}
This invariant identifies $\pi_3(S^2)$ with $\Z$ \cite{Hopf,Milnor}. In \cite{BFM25} we counted tetrahedra mapping onto a target triangle. These tetrahedra correspond to fibre edges, while our counted triangles correspond to fibre vertices. Since every fibre component is a circle, the two counts agree by \eqref{eq:circle-count}.

\begin{prop}\label{prop:hopf}
If $H(f)\ne0$, then $\mu(f,s)\ge9$ for every triangle $s$ of an arbitrary smooth target triangulation. For every $d\ne0$,
\[
 \mu(d)=9,\qquad \lambda(3,2,d)\ge12,
 \qquad \lambda(3,2,\pm1)=\lambda(3,2,\pm2)=12.
\]
\end{prop}

\begin{proof}
Choose $p\in\itr s$, and choose $q$ outside the three segments joining $p$ to the vertices of $s$. Lemma~\ref{lem:colors} deforms a component $\gamma$ of $f^{-1}(p)$ to each of its three color cycles in the complement of $f^{-1}(q)$. If one pivot count is at most two, \eqref{eq:short-cycle} gives $\lk(\gamma,f^{-1}(q))=0$. Since the sum of these linking numbers is nonzero, some component has all three pivot counts at least three. Equation~\eqref{eq:circle-count} proves the lower bound.

For the upper bound, take an unknotted circle in $S^3$ with normal framing $d$. A parallel copy has linking number $d$ with the core. Triangulate the circle by three edges and apply Lemma~\ref{lem:tube}; its count is $3+2\cdot3=9$, and \eqref{eq:hopf-link} identifies the class. This proves the Hopf part of Theorem~\ref{th:spherical}.

For the vertex bound the target is $S^2_4$. Fix any color $i\in\{0,1,2,3\}$, choose a target triangle $s_i$ containing $e_i$, and take $p_i\in\itr s_i$ and a regular value $q_i\notin[p_i,e_i]$. Write $M_i=f^{-1}(p_i)$, with all its oriented components. The color homotopy stays in the complement of $f^{-1}(q_i)$, so
\[
 H(f)=\lk(M_i,f^{-1}(q_i))
     =\lk((R_i)_*[M_i],f^{-1}(q_i)).
\]
Thus $(R_i)_*[M_i]$ is a nonzero integral edge cycle in the monochromatic simplicial complex $K_i$. Such a cycle requires at least three vertices. Repeating this argument for all four colors gives at least twelve source vertices. Here the target triangle and the fibre are chosen separately for each color.

The twelve-vertex representatives with Hopf invariant one \cite{MS} and two \cite{Mad} attain the bound. Their combinatorial three-sphere sources have smooth realizations with the same simplices: the low-dimensional smoothing construction starts from linear vertex stars (their two-sphere links are polytopal), extends over edges, faces, and tetrahedra, and gives the standard smooth structure \cite{Thurston}. Reflection changes the sign.
\end{proof}

\subsection{Convex polytopal Hopf classes}

\begin{thm}\label{th:polytopal}
There is a constant $A>0$ such that, for every $d\ne0$,
\begin{equation}\label{eq:polytopal-bounds}
 6+\lceil3\sqrt{|d|}\rceil\le\mu^*(d)\le A\sqrt{|d|}.
\end{equation}
Moreover, $\mu^*(d)\ge12$ if $|d|\ge2$, and $\mu^*(0)=0$.
\end{thm}

\begin{proof}[Lower bounds]
Let $f:\partial P\to S^2_4$ represent $d$, and fix a target triangle $ABC$. Write $m_X$ for the total pivot count of color $X$, over all fibre components. By \eqref{eq:circle-count},
\[
 r:=\mu(f,ABC)=m_A+m_B+m_C.
\]
Let $Z_X$ be the total color cycle from Lemma~\ref{lem:colors}. Its oriented edge occurrences do not depend on the point in $\itr(ABC)$. Collapse fibres at three points near $A,B,C$ along disjoint target segments. The linking formula gives
\begin{equation}\label{eq:three-links}
 \lk(Z_A,Z_B)=\lk(Z_A,Z_C)=\lk(Z_B,Z_C)=d,
\end{equation}
so every $m_X\ge3$.

Cancel opposite edge occurrences in $Z_X$ and decompose it into simple directed cycles, with multiplicities. A cycle of length $b$ bounds an affine fan of $b-2$ triangles. All these fans lie in $P$ by convexity. Their sum is a two-chain $F_X$ with $\partial F_X=Z_X$. Since $Z_X\ne0$, there is at least one cycle, so the total number of triangle terms is at most $m_X-2$.

The convex hulls of the vertices of two different colors are disjoint on $\partial P$. Indeed, their intersections with a supporting facet are disjoint faces of that simplex. Thus intersections between $F_X$ and $F_Y$ occur only in the interior of $P$. Portions of either chain on the boundary can be pushed inward without changing these intersections. After a small generic perturbation of the polytope vertices preserving its boundary complex, each pair of differently colored triangles meets at most once. The linking--intersection identity gives
\[
 |d|\le(m_A-2)(m_B-2),\quad
 |d|\le(m_A-2)(m_C-2),\quad
 |d|\le(m_B-2)(m_C-2).
\]
Multiplying and applying the arithmetic--geometric mean inequality yields $r\ge6+3\sqrt{|d|}$, proving the first bound.

For the bound twelve, observe that a triangular edge cycle and a disjoint triangular or quadrilateral edge cycle on $\partial P$ have linking number at most one in absolute value. For two triangles use their flat fillings. For a quadrilateral $\beta$, its four vertices generically span a tetrahedron $Q_\beta$, and two coherent boundary faces fill $\beta$. The triangular filling of the other cycle meets $Q_\beta$ in at most a segment. Its endpoints lie on $\partial Q_\beta$, since its boundary edges are disjoint from the convex hull of the other vertices. The endpoint intersections have opposite signs; selecting two faces gives intersection number $0$ or $\pm1$.

If $r\le11$, then, after permuting colors, $m_A=3$ and $m_B\le4$. After cancellation, the corresponding nonzero cycles are a triangle and a triangle or quadrilateral. Equation~\eqref{eq:three-links} then implies $|d|\le1$.
\end{proof}

For the upper bound, we need to refine a polytopal sphere while controlling the number of tetrahedra.

\begin{lemma}[Controlled polytopal refinement]\label{lem:poly-refinement}
Fix a simplicial map $h_0:\partial P_0\to S^2_4$, with $P_0$ a convex simplicial four-polytope having $t_0$ facets. Suppose $K$ is obtained from $S^2_4$ by stellar subdivisions in interiors of triangles of one original face. If $K$ has $t(K)$ triangles, there is a convex polytopal sphere $T_K$ and a simplicial map $h_K:T_K\to K$ representing $h_0$, with at most $15t_0t(K)$ tetrahedra.
\end{lemma}

\begin{proof}
Place the origins inside $P_0$ and the target tetrahedron. Extend $h_0$ homogeneously to $\widehat h_0:\R^4\to\R^3$, linear on each cone of the source face fan $\Sigma_0$. The target fan over $K$ is regular, since stellar subdivision preserves regularity. Let $\psi$ be a convex homogeneous PL function with precisely its cones as domains of linearity, and let $\varphi_0$ be the gauge of $P_0$. Intersect the old source cones with the inverse images of the target cones. On each old cone $\psi\widehat h_0$ is convex and bends strictly across every new wall. For sufficiently large $M$,
\[
 \varphi=M\varphi_0+\psi\circ\widehat h_0
\]
is positive away from the origin, globally convex, and strictly bent at all walls of this refinement: the old positive bends dominate the finitely many possible negative bends there.

Use a global vertex order to take a regular pulling triangulation with no new rays \cite{DRS10}. Equivalently, successively perturb the supporting heights by sufficiently small amounts, preserving the strict bends. The resulting positive convex homogeneous function $\varphi'$ has simplicial domains of linearity. The sublevel set $\{\varphi'\le1\}$ is a convex simplicial polytope realizing the required refinement; this is the usual support-function criterion for regular fans \cite{AI15}.

In an old tetrahedron mapping onto a target triangle, the inverse image of a triangle of $K$ is a triangular prism or a degeneration, with at most six vertices. Each vertex maps to a vertex of $K$, and a triangulation without new vertices has at most $\binom64=15$ tetrahedra. Old tetrahedra mapping to edges or vertices are unchanged. There are at most $t_0t(K)$ three-dimensional cells, which proves the count and supplies the required labeling. Identify $|K|$ with the original target sphere by the subdivision map. The source refinement and its radial realization then identify $h_K$ with a map homotopic to $h_0$.
\end{proof}

\begin{proof}[Upper bound in Theorem~\ref{th:polytopal}]
Fix a polytopal representative $h_0$ with $H(h_0)=1$; existence follows from the derived-subdivision argument in Section~\ref{sec:intro}. Let $t_0$ be its facet count. For $k\ge1$, construct a degree-$k$ map $g_k:K_k\to S^2_4$ as follows. Label the original target tetrahedron entirely by $D$. Make $k-1$ stellar subdivisions in one face $s_0$, still using $D$, and select $k$ of its $2k-1$ triangles. Inside each selected triangle insert successively vertices $A,B,C$, choosing each new triangle to contain the previously inserted vertices. This creates exactly one $ABC$ triangle in that disc, with boundary still labeled $D$. Interchange the first two labels if needed to make its orientation positive. Thus
\[
 \deg g_k=k,\qquad t(K_k)=4+2((k-1)+3k)=8k+2.
\]

Apply Lemma~\ref{lem:poly-refinement} and put $F_k=g_kh_{K_k}$. Recall that for $h:S^3\to S^2$ and a degree-$k$ map $g:S^2\to S^2$, bilinearity of linking gives
\begin{equation}\label{eq:hopf-postcomposition}
 H(g\circ h)=(\deg g)^2H(h)=k^2H(h).
\end{equation}
Indeed, choose two regular values of $g$ whose preimages are regular values of $h$. Their composite fibres are signed unions of fibres of $h$; every pair has linking number $H(h)$, and each sum of local signs is $\deg g$. This proves the formula, after smooth approximation if needed. In the present construction the two values in $ABC$ each have exactly $k$ positively oriented preimages, so
\[
 H(F_k)=k^2,\qquad
 \mu(F_k,ABC)\le15t_0(8k+2)\le150t_0k.
\]
There are also at least nine all-$D$ tetrahedral facets. Indeed, choose an unchanged target face. Its fibre under $h_0$ has at least nine vertices by Proposition~\ref{prop:hopf}, hence at least nine edges. The corresponding source tetrahedra are unchanged by the refinement, and $g_k$ labels all their vertices $D$.

Write $d>0$ as a sum of four squares $d=\sum_{j=1}^4 k_j^2$, omit zero terms, and form connected sums along all-$D$ facets. This operation is polytopal: projectively identify the chosen facets, put the polytopes on opposite sides, and flatten towards the common facet sufficiently to preserve all other supporting facets and convexity. It adds no vertices beyond the identified ones. The maps agree at the gluing facets; their Hopf invariants add and their counted faces are unchanged. Hence
\[
 \mu^*(d)\le150t_0\sum_j k_j\le300t_0\sqrt d.
\]
Reflection treats negative $d$, and the constant map treats zero.
\end{proof}

The order of growth in \eqref{eq:polytopal-bounds} is sharp, but the optimal constants and exact small values are open. For example, we do not know whether $\mu^*(2)=12$.

\section{Circle and surface fibres}\label{sec:short}

We first complete the proof of Theorem~\ref{th:spherical} for the groups $\pi_{n+1}(S^n)$. We then treat $\pi_{n+2}(S^n)$ for every $n\ge2$, using the surface case $n=2$ as the base of a vertex-link induction.

\subsection{Three spanning surfaces and the eight-edge circle}

For $n\ge3$, stabilization identifies $\pi_{n+1}(S^n)$ with the framed bordism group $\pi_1^S\cong\Z_2$. Its nonzero element $\eta_n$ is an iterated suspension of the Hopf class \cite{Hat,Putman}. We first give a criterion for a framed circle in $S^4$ to be null-bordant.

\begin{lemma}[Three spanning surfaces]\label{lem:three-surfaces}
Let a circle $\gamma\subset S^4$, framed by $(v_1,v_2,v_3)$, bound three compact connected orientable smooth surfaces $F_1,F_2,F_3$ with pairwise disjoint interiors. If their inward directions along $\gamma$, modulo its tangent, form a framing homotopic to the given one, then the framed circle is zero in $\pi_4(S^3)$.
\end{lemma}

\begin{proof}
Straighten the collars in the framed tube $\gamma\times D^3$ to the rays $t e_i$, $t\ge0$. Bend the collar of $F_2$ in the $(e_1,e_2)$-plane so that it starts along $-e_1$ and returns to the positive $e_2$ ray, avoiding the positive $e_1$ ray. Use the same curve over each point of $\gamma$. It remains disjoint from the $e_3$ collar, and
\[
 \Sigma=F_1\cup_\gamma(-F_2)
\]
is a closed smooth oriented surface whose normal plane along $\gamma$ is framed by $(v_2,v_3)$. No rotation along $\gamma$ has been introduced.

The normal Euler number of $\Sigma$ equals its self-intersection and is zero because $H_2(S^4;\Z)=0$. Fix a global trivialization of this oriented normal plane bundle. A small push-off $\gamma_{v_3}$ bounds the remainder of $F_3$ in $S^4\setminus\Sigma$, so its linking number with $\Sigma$ is zero. The push-off in a constant direction of the trivialization also has linking number zero, since it bounds the corresponding push-off of $F_1$. The difference is the winding number of $v_3$ in the normal trivialization: each turn adds a meridian of $\Sigma$. Thus this winding number is zero.

In the fixed trivialization, the boundary field $v_3$ is a map $\gamma\to S^1$ of degree zero. It is therefore homotopic to a constant and extends over $F_1$, even when $F_1$ has positive genus. Complete the extension to an oriented normal two-framing, whose boundary framing is homotopic to $(v_2,v_3)$. To obtain a framed null-bordism, move the interior of $F_1$ into $S^4\times(0,1]$. Near the boundary use inward coordinate $r^2$ and height $r$; this makes the surface meet the bottom boundary vertically. The two normal fields extend, and the normal to the height graph supplies a third field restricting, with the chosen sign, to $v_1$. Its boundary is the prescribed framed circle.
\end{proof}

\begin{lemma}[A fibre with eight edges]\label{lem:eight}
For a simplicial map $S^4\to S^3$ between smooth triangulations, a fibre component over a target tetrahedron with eight edges and four nonconstant coordinates has zero framed class.
\end{lemma}

\begin{proof}
Each coordinate changes at least twice, so each changes exactly twice and uses precisely two source vertices $u_i,v_i$. Its two pivot edges traverse $[u_i,v_i]$ in opposite directions. Apply the homotopy $H_i$ of Lemma~\ref{lem:colors} to the circle $\gamma$. For parameter $s<1$ the barycentric multipliers are positive and invertible, giving an embedded circle at every level; distinct levels have distinct target images. At $s=1$ the two $i$-edges map onto $[u_i,v_i]$, while the intervening arcs collapse to its endpoints. The mapping cylinder of this fold onto an interval is a disc. The image of $H_i$ is therefore an embedded polyhedral disc $D_i$ with boundary $\gamma$.

This disc is locally flat except possibly at $u_i,v_i$. Away from its terminal edge use the product in Lemma~\ref{lem:product}; transverse to the interior of that edge it is a pair of rays in three-space. At each endpoint choose a small four-ball $B$ disjoint from $\gamma$, so that $D_i\cap B$ is a cone on a knot $J\subset\partial B\cong S^3$. The oriented knot $J$ bounds a compact connected orientable Seifert surface $S\subset\partial B$: one may construct it from oriented Seifert circles of a knot diagram by attaching the crossing bands. Push the interior of $S$ into a collar of $\partial B$ inside $B$, keeping its boundary $J$ fixed, and make it properly embedded near $J$. Replace $D_i\cap\itr B$ by this pushed-in surface and round the seam along $J$. The orientations agree along the seam. The remaining outer boundary is still exactly $\gamma$; the genus may have increased.

For distinct colors the interiors of the discs lie over the disjoint segments $(y,e_i]$. Their endpoints are distinct source vertices, so all endpoint balls may be chosen disjoint and avoiding the other discs. Away from a common collar of $\gamma$, the remaining compact pieces have disjoint neighbourhoods; smooth them there. This gives orientable surfaces $F_i$ with pairwise disjoint interiors. Near $\gamma$ retain the ray collars in one smoothed product neighbourhood. Their inward directions are $e_i-y$ in the target normal coordinates. Any three are independent and, in a suitable order, differ from the original frame by a constant positive linear transformation. Lemma~\ref{lem:three-surfaces} now applies. The replacements change only the spanning surfaces, leaving the map and its eight-edge fibre unchanged. The surfaces may have positive genus, which is why the argument requires spanning surfaces rather than spanning discs.
\end{proof}

\begin{prop}\label{prop:circle-groups-lower}
For $n\ge3$, a simplicial map $f:S^{n+1}\to S^n$ between smooth triangulations with at most eight source $n$-faces over some target facet is null-homotopic.
\end{prop}

\begin{proof}
Each fibre component has at most eight edges. For $n=3$, a constant coordinate gives, by Lemma~\ref{lem:link-model}, a map $S^3\to S^2$ with at most eight counted triangles. Proposition~\ref{prop:hopf} makes its Hopf invariant zero, so the component has zero stable framed class. If no coordinate is constant, there are four active coordinates and exactly eight edges; Lemma~\ref{lem:eight} gives the same conclusion.

For $n\ge4$, every component has a constant coordinate, since otherwise it would need at least $2(n+1)\ge10$ edges. The link model reduces to a map $S^n\to S^{n-1}$ with at most eight counted faces, which is null by induction. In every case the stable framed classes of all components vanish. Their sum represents the stabilization of $[f]$, and stabilization is injective on $\pi_{n+1}(S^n)$ for $n\ge3$. An empty fibre also makes $f$ null, since its image misses a point.
\end{proof}

\subsection{A geometric representative attaining nine}

To obtain sharp examples for balls, we construct a geometric representative of the Hopf class and then suspend it.

\begin{lemma}[Geometric suspension]\label{lem:suspension}
Let $h:K\to\partial\Delta^{q+1}$ be simplicial on a geometric subdivision of the boundary of a convex $(p+1)$-polytope. There is such a map in one higher dimension representing $\Sigma[h]$, with
\[
 \mu(h^+,[q+2]*s)=\mu(h,s)
\]
for every old target facet $s$.
\end{lemma}

\begin{proof}
Take a bipyramid over the source polytope. Keep the old labels, give its upper apex the new label $q+2$, and its lower apex the old label $0$. The upper half is the cone on $h$; the lower half extends $h$ into the old target simplex. All extensions into that convex simplex are homotopic relative to the equator, so this represents the suspension. Only the upper apex carries the new label; counted faces are exactly its cones over the old counted faces. Radial projection gives a smooth realization as well.
\end{proof}

\begin{lemma}[A geometric Hopf representative]\label{lem:hopf-seed}
There is a geometric subdivision of a convex polytope boundary representing the Hopf class with exactly nine triangles over one target facet.
\end{lemma}

\begin{proof}
Take triangles $[a_0,a_1,a_2]$ and $[b_0,b_1,b_2]$ in complementary planes in $\R^4$, each surrounding the origin. The boundary of their convex hull is the join of their boundary cycles. Put $c_{ij}=(a_i+b_j)/2$, with indices modulo three. In an original tetrahedron write
\[
 a=a_i,\quad a'=a_{i+1},\quad b=b_j,\quad b'=b_{j+1},
 \qquad (p,q,t,u)=(c_{ij},c_{i+1,j},c_{i,j+1},c_{i+1,j+1}).
\]
Subdivide it into
\begin{equation}\label{eq:seed-tetrahedra}
 \begin{gathered}\relax
 [a,a',q,u],\quad[a,p,q,u],\quad[a,p,t,u],\\
 [b,b',t,u],\quad[b,p,t,u],\quad[b,p,q,u].
 \end{gathered}
\end{equation}
The middle plane cuts the original tetrahedron into two triangular prisms, and these are compatible triangulations with middle diagonal $pu$. They agree on the old faces. Set
\[
 L(a_i)=0,\qquad L(b_j)=3,\qquad L(c_{ij})=i-j\pmod3.
\]
Every tetrahedron misses a color. For each $j$ the triangles over $[1,2,3]$ are exactly
\[
 [b_j,c_{j+2,j+1},c_{j,j+1}],\quad
 [b_j,c_{j+2,j},c_{j+2,j+1}],\quad
 [b_j,c_{j+1,j},c_{j+2,j}],
\]
giving nine in total.

To identify the class, temporarily give the $a_i$ a separate label $N$ and the $b_j$ a label $S$. The target is then the suspension of the three-edge circle. On the middle torus the equatorial map is homotopic to $(\theta,\varphi)\mapsto\theta-\varphi$: it has degrees $1$ and $-1$ on the two factors. On the two solid tori it extends into the respective target discs, sending the cores to the poles. After identifying the equatorial maps by this homotopy, these extensions are homotopic relative to the middle torus to those of the Hopf map, since the target discs are contractible. Contracting the target edge $N0$ gives a degree-one map to $\partial\Delta^3$ and the displayed four-color labeling. Choose orientations so that the Hopf invariant is one.
\end{proof}

\begin{proof}[Consecutive-dimension assertion of Theorem~\ref{th:spherical}]
The Hopf case is Proposition~\ref{prop:hopf}. For $n\ge3$, iterate Lemma~\ref{lem:suspension} on Lemma~\ref{lem:hopf-seed}. The resulting class is $\eta_n$ and the count is nine. Proposition~\ref{prop:circle-groups-lower} proves the reverse inequality.
\end{proof}

\subsection{Surface fibres}

For every $n\ge2$, the groups and suspension maps satisfy
\begin{equation}\label{eq:surface-groups-suspension}
 \pi_{n+2}(S^n)\cong\Z_2,\qquad
 \Sigma a_n=a_{n+1},\qquad
 \Sigma:\pi_{n+2}(S^n)\xrightarrow{\ \cong\ }\pi_{n+3}(S^{n+1}).
\end{equation}
This includes $n=2,3$, before the usual stable range. Stabilization identifies each group with $\pi_2^S\cong\Z_2$; see the classical treatment in \cite[\S15]{Pon59} and the modern exposition in \cite[Theorem~10.5]{Putman}.

A framed regular surface $N\subset S^{n+2}$ induces a quadratic refinement
\[
 q:H_1(N;\Z_2)\to\Z_2,\qquad q(x+y)=q(x)+q(y)+x\cdot y.
\]
The sum of the component Arf invariants determines its stable framed class \cite[\S15]{Pon59}; see also \cite[Lemmas~10.3--10.4 and Theorem~10.5]{Putman}. We will use the following framed torus.

\begin{lemma}[A torus with Arf invariant one]\label{lem:framed-torus}
For every $n\ge2$, there is a smoothly embedded normally framed torus in $S^{n+2}$ with Arf invariant one, representing $a_n$ under Pontryagin--Thom.
\end{lemma}

\begin{proof}
In $\R^4\subset S^4$, take the product torus
\[
 x(\theta,\varphi)=(\cos\theta,\sin\theta,\cos\varphi,\sin\varphi),
 \qquad \theta,\varphi\in\R/2\pi\Z.
\]
It has the orthonormal normal frame
\[
 u=(\cos\theta,\sin\theta,0,0),\qquad
 v=(0,0,\cos\varphi,\sin\varphi).
\]
For this product frame the two coordinate circles $\alpha,\beta$ have $q_0(\alpha)=q_0(\beta)=0$: their induced circle framings are the radial framing of a planar circle together with two constant transverse directions. Each bounds its planar disc after bending the disc into the bordism direction.

Put $\psi=\theta+\varphi$ and replace the normal frame by
\begin{equation}\label{eq:torus-framing}
 \nu_1=\cos\psi\,u+\sin\psi\,v,\qquad
 \nu_2=-\sin\psi\,u+\cos\psi\,v.
\end{equation}
Along either coordinate circle this makes one full rotation of the normal plane. A full twist changes the induced framed-circle class by the nonzero element of $\pi_1^S\cong\Z_2$ (equivalently, the loop of rotations is a generator of $\pi_1(SO(3))$). Thus $q(\alpha)=q(\beta)=1$. Since $\alpha\cdot\beta=1$ modulo two, $\Arf(q)=q(\alpha)q(\beta)=1$.

For $n>2$, embed $\R^4$ as the first four coordinates of $\R^{n+2}$ and append the constant normal vectors $e_5,\ldots,e_{n+2}$ to \eqref{eq:torus-framing}. This is stabilization of the framing; it preserves the quadratic refinement and its Arf invariant. Pontryagin--Thom identifies it with the iterated suspension of the codimension-two class. By \eqref{eq:surface-groups-suspension}, these nonzero classes are precisely $a_n$.
\end{proof}

\begin{lemma}[Three coordinates and eight vertices]\label{lem:cube}
Suppose a graph as in Lemma~\ref{lem:graph-bound} has exactly three nonconstant coordinates, minimum degree at least three, and at most eight vertices. Then it is the cube graph on all eight binary triples.
\end{lemma}

\begin{proof}
Add all edges between the selected tuples that differ in one coordinate, obtaining a connected graph $\Gamma$ of minimum degree at least three. By \eqref{eq:alphabet-bound}, the numbers of values taken by the three coordinates, up to permutation, are
\[
 (2,2,2),\quad(3,2,2),\quad(4,2,2),\quad(3,3,2).
\]
In the binary case every vertex has at most three possible neighbours; the degree condition forces all of them, hence the whole cube and every cube edge of the original graph.

For $(r,2,2)$ with $r=3$ or $4$, group the tuples in four columns indexed by the binary pair. A vertex in a column of size $b$ has $b-1$ neighbours there and at most two outside, so every nonempty column has size at least two. At least three columns are nonempty, by connectivity and activity of both binary coordinates. Three columns form a path; its end columns need at least three vertices, forcing sizes $3,2,3$. Each end vertex then needs a match in the middle column, which cannot accommodate three values. With four columns, each has size two and both outside matches are compulsory. All columns use the same two first-coordinate values, again a contradiction.

For $(3,3,2)$, the two binary layers are subsets of a $3\times3$ rook graph. Each layer has minimum internal degree two, since a vertex has at most one neighbour in the other layer. A three-vertex layer is a row or column clique and requires all three matches in the other layer. That layer has at most five vertices; an extra vertex has at most one neighbour in the clique, one among the extras, and none in the first layer, contradicting degree three. Without extras a coordinate is constant. Otherwise both layers have four vertices. Four rook positions of minimum degree two form a $2\times2$ rectangle: a row or column of three would leave the fourth point with degree at most one, and in its absence every point needs both a row and a column companion. All matches across the layers are then compulsory. The identical rectangles use only two values of each coordinate, contradicting $(3,3,2)$.
\end{proof}

\begin{proof}[Theorem~\ref{th:surface-groups}]
\emph{Lower bound: the case $n=2$.} Suppose $f:S^4\to S^2_4$ has at most eight triangles over a target facet. By \eqref{eq:cell}, its fibre is a closed surface with triangular or quadrilateral cells. The graph of each component $N$ is two-connected and has minimum degree at least three.

If all three coordinates are active, Lemma~\ref{lem:cube} gives the cube graph. Every cell is a quadrilateral, since binary coordinates exclude triangular factors in \eqref{eq:cell}. Hence
\[
 f_0(N)=8,\qquad f_1(N)=12,\qquad f_2(N)=6.
\]
Thus $N$ is a sphere and its Arf invariant is zero. If a coordinate is constant, Lemma~\ref{lem:link-model} gives a link map $S^3\to S^1$ whose framed fibre, with one normal direction added, represents $N$. Since $\pi_3(S^1)=0$, this stable class is also zero. The total Arf invariant vanishes, so injectivity of stabilization implies $[f]=0$.

\emph{Lower bound: induction on $n$.} Let $n\ge3$, and suppose the assertion is known for $n-1$. If $f:S^{n+2}\to S^n_{n+2}$ has at most eight counted $n$-faces, every component $N$ of its surface fibre has at most eight vertices. If all $n+1$ coordinates were active, \eqref{eq:strict-count} would give
\[
 |\vrt(N)|\ge2(n+1)+1\ge9.
\]
Thus a source vertex is constant on $N$. Apply Lemma~\ref{lem:link-model} to this single component. Its link map
\[
 h_N:S^{n+1}\longrightarrow S^{n-1}_{n+1}
\]
has exactly $|\vrt(N)|\le8$ counted faces and is null-homotopic by induction. The stable framed class of $N$ is therefore zero. Summing over the components and using \eqref{eq:surface-groups-suspension} gives $[f]=0$. An empty fibre also makes $f$ null, since its image misses a point. This proves the lower bound for every $n\ge2$.

\emph{Upper bound.} Take the framed torus of Lemma~\ref{lem:framed-torus} in codimension $n$. On vertices $\Z/7$, the faces
\[
 \{i,i+1,i+3\},\qquad\{i,i+2,i+3\}\qquad(i\in\Z/7)
\]
give a seven-vertex torus: each edge belongs to two faces, each vertex link is a six-cycle, and the orientations $(i,i+1,i+3)$, $(i,i+3,i+2)$ are coherent. Its face vector is $(7,21,14)$ and its Euler characteristic is zero. More explicitly, it is the quotient of the integer square grid, split along the diagonals of direction $(1,1)$, by the sublattice
\[
 \Lambda=\{(a,b)\in\Z^2:a+2b\equiv0\pmod7\}.
\]
A linear identification of $\R^2/\Lambda$ with the parameter torus above gives a smooth realization with these same simplices and the specified normal framing. Lemma~\ref{lem:tube}, with $r=n$, gives
\[
 \mu_n\le7+21n+14\binom n2=7(n+1)^2.
\]
The eventual stabilization assertion is proved in Proposition~\ref{prop:surface-groups-stabilization}.
\end{proof}

The tube construction gives $\mu_2\le63$ in the smooth category. To apply the geometric suspension of Lemma~\ref{lem:suspension}, one needs a geometric representative. We use such a representative in Section~\ref{sec:balls} to obtain a uniform bound, but do not control its count by $63$.

\section{The quantitative Sperner lemma for geometric balls}\label{sec:balls}

Let $P\subset\R^{m+1}$ be a compact convex polytope with nonempty interior, and let $T$ be a geometric triangulation of $P$. Assume that $f=f_L:P\to\Delta^{n+1}$ maps $\partial P$ into $\partial\Delta^{n+1}$, and put
\[
 a=[\partial f],\qquad C=C(T,L),\qquad r=n+1,\qquad k=m-n.
\]
For $z\in\itr\Delta^r$, the fibre $M=f^{-1}(z)$ is closed and has exactly $C$ vertices. The pyramid construction first gives a representative of the suspended boundary class with the same count. We then recover the original class when every fibre component has a constant vertex. The graph bound ensures this condition when the count is small. The Hopf case requires a separate comparison of interior and boundary color cycles.

The cases $m=n$ and $n=1<m$ follow from \eqref{eq:relative-degree} and $\pi_m(S^1)=0$, respectively. We use the addition lemma only when $m>n\ge2$.

\subsection{The pyramid comparison}

\begin{prop}[Comparison with the suspended class]\label{prop:pyramid}
Under the hypotheses of Theorem~\ref{th:sperner}, there is a geometric subdivision of a convex polytope boundary representing $\Sigma a$ with exactly $C(T,L)$ faces over one target facet. Consequently,
\[
 C(T,L)\ge\mu(m+1,n+1,\Sigma a).
\]
\end{prop}

\begin{proof}
Place $P$ in $\R^{m+1}\times\{0\}$ and choose an apex $v$ outside this hyperplane. The convex pyramid $Q=\conv(P,v)$ has boundary triangulation
\[
 \widehat T=T\cup(v*\partial T),\qquad
 |\widehat T|=\partial Q\cong S^{m+1}.
\]
Keep the old labels and give $v$ the new label $n+2$. They define a simplicial map
\[
 \widehat f:\widehat T\longrightarrow\partial\Delta^{n+2}.
\]
Indeed, a base simplex maps into the old facet $s=[0,\ldots,n+1]$, while a side simplex maps into $e_{n+2}*\partial s$ and misses an old color. The $(n+1)$-faces over $s$ are precisely the fully labeled faces of the original $T$: a side face containing $v$ has the new label, and a side face without $v$ belongs to $\partial T$. Hence
\begin{equation}\label{eq:pyramid-count}
 \mu(\widehat f,s)=C(T,L).
\end{equation}

Regard $P$ topologically as the cone on $\partial P$ from an interior point. Its map $f$ is homotopic relative to $\partial P$ to the cone extension of $\partial f$ into $s$, by a straight-line homotopy in this convex simplex. On $v*\partial P$ the map is already the cone on $\partial f$. Under the identifications
\[
 \partial Q\cong\Sigma\partial P,\qquad
 \partial\Delta^{n+2}=s\cup(e_{n+2}*\partial s)\cong\Sigma\partial s,
\]
the resulting map represents $\Sigma a$, with compatible orientations. Radial projection from an interior point of $Q$ gives a smooth realization on the standard sphere with the same faces and labels. Equation~\eqref{eq:pyramid-count} proves the inequality.
\end{proof}

\subsection{Componentwise localization}

\begin{prop}\label{prop:localization}
Suppose $m>n\ge1$. If every component of $M$ has a constant source vertex, then
\[
 \mu(m,n,a)\le C.
\]
Different components may have different constant vertices and different constant colors.
\end{prop}

\begin{proof}
The case $n=1$ is immediate from $\pi_m(S^1)=0$, so assume $n\ge2$. Add an exterior collar by enlarging $P$ homothetically and triangulating its boundary prisms with a common vertex order. Copy the labels to the outer vertices. Each prism misses the same color as its base simplex, so no full face or fibre is added, the boundary class is unchanged, and every old vertex is now interior.

Choose a constant vertex on each component, and group the components assigned to the same vertex $v$ into $M_v$. Let $c=c(v)$ and let $C_v$ be their total fibre vertex count. Lemma~\ref{lem:link-model} gives a map
\[
 h_v:S_v'\to\partial\Delta^r,\qquad
 S_v'\cong S^m,\qquad \mu(h_v,s_c)=C_v,
\]
where $s_c$ is the facet opposite $c$. In the notation of that lemma, its retained subcomplex $Q_v$ is the union of the closed stars of the faces obtained by deleting $v$ from the counted faces. Radial projection gives the link its smooth realization.

We will show that the link classes sum to the boundary class in $\pi_m(S^n)$. Form the straight cones
\[
 K_v=\{sv+(1-s)x:x\in M_v,\ 0\le s\le1\}.
\]
They lie in the corresponding vertex stars. With $\lambda_v$ denoting the global barycentric coordinate of $v$,
\begin{equation}\label{eq:cone-coordinates}
 f(sv+(1-s)x)=se_c+(1-s)z,\qquad
 \lambda_v=f_c\ge z_c>0\quad\text{on }K_v.
\end{equation}
These cones are pairwise disjoint. For different colors the target segments meet only at $z$, where the groups of components are disjoint. For distinct vertices $v,w$ of the same color, an intersection would imply $f_c\ge\lambda_v+\lambda_w=2f_c>f_c$.

Except for its apex, $K_v$ lies in the interior of $v*Q_v$, since all cofaces of its fibre carriers were retained. Thus
\[
 U_v=\itr(v*Q_v)\ \cup\
 \{x\in\itr\st_T(v):\lambda_v(x)>z_c\}
\]
is a neighbourhood of $K_v$ in the interior of the star. Choose disjoint smooth balls $B_v$, star-shaped about $v$, with
\[
 K_v\subset\itr B_v\subset B_v\subset U_v\cap\itr P.
\]
Here is a concrete construction. Each $K_v$ is a finite union of convex polytopes containing $v$. Add a sufficiently small centered cube to each; their union has a ball in its kernel and a continuous positive radial function. A smooth radial function between two such thickenings gives $B_v$. Compactness and disjointness of the cones allow the choices to remain disjoint and inside the specified neighbourhoods.

Extend $h_v$ conically over the original star:
\[
 F_v(tv+(1-t)y)=te_c+(1-t)h_v(y),\qquad y\in S_v'.
\]
It agrees with $f$ on $v*Q_v$, and everywhere $(F_v)_c\ge\lambda_v$ and $f_c\ge\lambda_v$. Its only fibre over $z$ is $M_v$, because a link simplex with all colors other than $c$ lies in $Q_v$. On $B_v$ the two maps therefore either agree or both have $c$-coordinate greater than $z_c$. Their straight-line homotopy avoids $z$ on $\partial B_v$.

Let $\rho_z:\Delta^r\setminus\{z\}\to\partial\Delta^r$ be radial retraction. The fibre $F_v^{-1}(z)=M_v$ lies inside $B_v$, so $F_v$ avoids $z$ on the radial annulus between $\partial B_v$ and the outer link $S_v'$. On the outer link, $F_v=h_v$ already takes values in $\partial\Delta^r$. Retraction of this annulus, together with the homotopy on $\partial B_v$, gives
\[
 [\rho_z f|_{\partial B_v}]
 =[\rho_z F_v|_{\partial B_v}]=[h_v]\quad\text{in }\pi_m(S^n).
\]
The map $\rho_z f$ is defined on the ball with the interiors of all $B_v$ removed. On its outer boundary it equals $\partial f$, and on each inner sphere it has class $[h_v]$. The boundary relation for this punctured ball therefore gives
\begin{equation}\label{eq:boundary-sum}
 a=\sum_v[h_v].
\end{equation}
Here each inner sphere is oriented as the boundary of the removed ball. Joining these spheres by thin tubes along a tree gives their homotopy sum, with the orientation opposite to their orientation as boundary components of the complement. Since $S^n$ is simply connected for $n\ge2$, the choices of basepoint paths do not affect the sum. Thus \eqref{eq:boundary-sum} holds in the original homotopy group, without stabilization. Lemma~\ref{lem:addition} now gives
\[
 \mu(m,n,a)\le\sum_v C_v=C.
\]
The same argument with an empty collection of balls covers an empty fibre.
\end{proof}

The hypothesis can be checked from the labeled triangulation. Form the graph whose vertices are the fully labeled $r$-faces and whose edges correspond to the $(r+1)$-simplices carrying all $r+1$ colors. Its components are the fibre components, and the condition is that all source faces in each component share a vertex. Barycentric rescaling shows that the condition is independent of $z\in\itr\Delta^r$.

\subsection{The Hopf boundary case}

In dimension four, a homological comparison of color cycles gives the required bound without a constant-vertex hypothesis. We use the following standard duality identity.

\begin{lemma}[Interior and boundary fibres]\label{lem:boundary-fibres}
Let $g:(V,\partial V)\to(D^2,S^1)$ be a PL map from a compact oriented PL three-manifold. For regular values $p\in\itr D^2$ and $q\in S^1$, in the product-neighbourhood sense, one has
\begin{equation}\label{eq:boundary-fibre-homology}
 [g^{-1}(p)]=\pm j_*[(g|_{\partial V})^{-1}(q)]\quad\text{in }H_1(V;\Z),
\end{equation}
where $j:\partial V\hookrightarrow V$ and fibres have the induced orientations. Disconnected $V$, including closed components, is allowed.
\end{lemma}

\begin{proof}
For a generator $\eta\in H^1(S^1)$, its connecting image is the relative orientation class $u\in H^2(D^2,S^1)$. Naturality gives $g^*u=\delta_V((g|_{\partial V})^*\eta)$. Poincar\'e--Lefschetz duality and the cap-product boundary identity give
\[
 \operatorname{PD}_V(\delta_V\alpha)
 =\pm j_*\operatorname{PD}_{\partial V}(\alpha).
\]
The dual classes are the regular fibre classes, proving \eqref{eq:boundary-fibre-homology}; see \cite[Section~3.3]{Hat}. On a closed component both classes are zero.
\end{proof}

\begin{prop}[Four-dimensional balls]\label{prop:hopf-ball}
For a geometric triangulation of a convex four-polytope, or a smooth triangulation of the standard $D^4$, a labeling with no fully labeled boundary tetrahedron and with $H(\partial f_L)\ne0$ has at least nine fully labeled tetrahedra.
\end{prop}

\begin{proof}
Write $d=H(\partial f)$ and choose $z\in\itr\Delta^3$. The fibre $M=f^{-1}(z)$ is a closed oriented union of circles; it is nonempty because otherwise radial retraction would extend $\partial f$ to $S^2$. Let $m_i$ be the total pivot counts for the four colors. By \eqref{eq:circle-count},
\begin{equation}\label{eq:four-colors}
 C=m_0+m_1+m_2+m_3.
\end{equation}
Let $Z_i=(R_i)_*[M]$ in the monochromatic subcomplex $K_i$. Lemma~\ref{lem:colors} gives
\begin{equation}\label{eq:four-short}
 m_i\le2\quad\Longrightarrow\quad Z_i=0.
\end{equation}

The coordinate level $V_i=f_i^{-1}(z_i)$ is a compact properly embedded PL three-manifold, by Lemma~\ref{lem:product} applied to $f_i:T\to[0,1]$. Its target section $E_i=\{t\in\Delta^3:t_i=z_i\}$ is a triangle, and restriction defines
\[
 g_i:(V_i,\partial V_i)\to(E_i,\partial E_i).
\]
The product charts for $f$ and $\partial f$ restrict to charts for the interior and boundary fibres of $g_i$. For $y\in\partial E_i$ in the interior of a target face, Lemma~\ref{lem:boundary-fibres}, followed by $R_i|_{V_i}$, therefore gives
\begin{equation}\label{eq:color-boundary}
 [Z_i]=\pm[b_i(y)]\quad\text{in }H_1(K_i),\qquad
 b_i(y)=(R_i)_*[(\partial f)^{-1}(y)].
\end{equation}
The latter cycle lies in $K_i\cap\partial T$.

For distinct colors $i,j$, choose a target face with vertices $i,j,h$ and put
\[
 y=z_i e_i+\varepsilon e_j+(1-z_i-\varepsilon)e_h,
 \qquad q=\varepsilon e_i+z_j e_j+(1-z_j-\varepsilon)e_h,
\]
where $0<\varepsilon<\min\{z_i,z_j,1-z_i,1-z_j\}$. The segments $[y,e_i]$ and $[q,e_j]$ are disjoint: the $i$-coordinate on the first is at least $z_i$, and on the second at most $\varepsilon$. Thus the two color homotopies preserve linking, and
\begin{equation}\label{eq:boundary-link}
 |\lk(b_i(y),b_j(q))|=|d|>0.
\end{equation}

If both $[Z_i]$ and $[Z_j]$ vanished in their monochromatic subcomplexes, \eqref{eq:color-boundary} would supply integral two-chains there bounding $b_i(y),b_j(q)$. The subcomplexes $K_i,K_j$ are disjoint. Push the chains into the interior of the ball relative to their boundaries, in disjoint neighbourhoods, and put them in general position. Their intersection number is zero, whereas the linking--intersection identity in a four-ball identifies it with \eqref{eq:boundary-link}, a contradiction.

At most one of the four classes $[Z_i]$ can therefore vanish. At least three pivot counts are at least three by \eqref{eq:four-short}, and \eqref{eq:four-colors} gives $C\ge9$. The proof uses the PL fibre structure and the topology of the ball, so applies in both stated categories.
\end{proof}

\subsection{Proof of the quantitative Sperner lemma}

\begin{proof}[Theorem~\ref{th:sperner}]
Proposition~\ref{prop:pyramid} proves \eqref{eq:sperner-suspended}. For \eqref{eq:sperner-comparison}, the case $a=0$ follows from $\mu(m,n,0)=0$. We consider the three stated conditions.

\emph{Equal dimensions.} For $m=n$, use the signed count \eqref{eq:relative-degree} and the degree part of Theorem~\ref{th:spherical}.

\emph{Consecutive dimensions.} Let $m=n+1$. The case $n=1$ has no nonzero class. For $n=2$, Proposition~\ref{prop:hopf-ball} and $\mu(d)=9$ give the comparison. For $n\ge3$, Theorem~\ref{th:spherical} gives $\mu(m,n,a)=9$, so it suffices to exclude $C\le8$. Every fibre circle would then have at most eight edges. With $n+2\ge5$ colors, a nonconstant coordinate of every color would require at least $2(n+2)>8$ edges. Each component thus has a constant vertex, and Proposition~\ref{prop:localization} gives $C\ge\mu=9$, a contradiction.

\emph{Small count.} It remains to assume $m>n+1$ and $C\le2(n+2)$. Each component has dimension $k=m-n\ge2$. By \eqref{eq:strict-count}, if all $n+2$ coordinates were active on a component, it would have at least $2(n+2)+1$ vertices. Hence every component has a constant vertex. Proposition~\ref{prop:localization} again gives $C\ge\mu(m,n,a)$, including the endpoint $C=2(n+2)$.
\end{proof}

The last argument also gives a componentwise criterion: when $m>n+1$, it suffices that each component have at most $2(n+2)$ fibre vertices, even if the total count is larger.

\begin{proof}[Corollary~\ref{cor:nine} and sharpness in consecutive dimensions]
For $m=n+1$, combine Theorems~\ref{th:spherical} and~\ref{th:sperner}. For $m=n+2$, equation~\eqref{eq:surface-groups-suspension} gives $\Sigma a_n=a_{n+1}\ne0$. Proposition~\ref{prop:pyramid} and Theorem~\ref{th:surface-groups} yield
\begin{equation}\label{eq:surface-groups-ball}
 C(T,L)\ge\mu_{n+1}\ge9.
\end{equation}
Alternatively, if $C\le8$, then $C\le2(n+2)$ and the small-count case of Theorem~\ref{th:sperner} would imply $C\ge\mu_n\ge9$.

For sharp geometric examples when $m=n+1$, cone the representative in Lemma~\ref{lem:hopf-seed}, or one of its geometric suspensions, to an interior vertex of color $0$. The fully labeled $(n+1)$-faces are precisely the cones over the nine boundary $n$-faces whose labels omit $0$. The boundary class is nonzero, so the bound nine is attained for every $n\ge2$.
\end{proof}

Sharpness for geometric balls with $m=n+2$ remains open. Even if $\mu_2=9$, a sharp five-ball example would require a geometric representative with that count; the definition of $\mu_2$ allows smooth triangulations.

\subsection{Eventual stabilization of the spherical minima}

Choose a geometric representative of $a_2$ on a subdivision of the boundary of a convex five-polytope, and let $B$ be its number of triangles over a chosen target facet. Such a representative exists by simplicial approximation on sufficiently fine geometric subdivisions of a simplex boundary. Fix this finite integer $B$; necessarily $B\ge9$.

\begin{prop}\label{prop:surface-groups-stabilization}
For this choice of $B$, the following hold.
\begin{enumerate}
\item For all $n\ge2$, one has $\mu_n\le B$. If $n\ge3$ and $B\le2n+2$, then $\mu_{n-1}\le\mu_n$. Consequently, $(\mu_n)$ is eventually constant.
\item If $n\ge2$ and $B\le2n+5$, every geometric ball satisfying the hypotheses of Theorem~\ref{th:sperner} with $m=n+2$ and $a=a_n$ satisfies
\[
 C(T,L)\ge\mu_n.
\]
\end{enumerate}
\end{prop}

\begin{proof}
Iterate Lemma~\ref{lem:suspension} on the chosen geometric representative. Equation~\eqref{eq:surface-groups-suspension} identifies the classes with $a_n$, and the count stays $B$. This proves the uniform bound.

For the monotonicity assertion, choose a minimizing representative of $a_n$, whose fibre has $\mu_n\le B$ vertices. If $B\le2n+2$, equation~\eqref{eq:strict-count} forces a constant coordinate on every component. At least one component $N$ has Arf invariant one. Its vertex-link map represents the nonzero class $a_{n-1}$, by Lemma~\ref{lem:link-model} and injectivity of stabilization, and has at most $\mu_n$ counted faces. Thus $\mu_{n-1}\le\mu_n$. The integer sequence is eventually nondecreasing and bounded above by $B$, so it is eventually constant.

For the second assertion, suppose $C<\mu_n$. Integrality and the uniform bound give
\[
 C\le\mu_n-1\le B-1\le2n+4=2(n+2).
\]
The small-count case of Theorem~\ref{th:sperner} then gives $C\ge\mu_n$, a contradiction.
\end{proof}

For a geometric representative with count $B$, the comparison holds for
$n\ge\max\{2,\lceil(B-5)/2\rceil\}$, and the sequence is nondecreasing from
$n\ge\max\{3,\lceil(B-2)/2\rceil\}$ onward. The latter bound does not locate the last increase of $(\mu_n)$. Since we have not computed $B$, neither conclusion gives a numerical dimension here.

The main open questions are the exact values of $\mu_n$, sharpness of the bound nine for geometric balls, and the comparison $C\ge\mu_n$ in the remaining dimensions. In particular, eventual constancy does not settle Conjecture~\ref{conj:nine}.

\medskip 

\noindent {\bf AI use statement.} LLMs were used for proof refinement, literature review, and manuscript revision.

 \medskip

 O. R. Musin,  University of Texas Rio Grande Valley, School of Mathematical and
 Statistical Sciences, One West University Boulevard, Brownsville, TX 78520, USA.

 {\it E-mail address:} oleg.musin@utrgv.edu

\end{document}